\documentclass{amsart}
\usepackage[hyperref,numberbysection,all,jump,Rn]{paper_diening}

\usepackage{mathtools}
\usepackage{pgfplots}
\pgfplotsset{compat=1.16}

\usepackage{tikz}
\usetikzlibrary{patterns}
\usetikzlibrary{calc}
\usepackage[thinlines]{easytable}
\usepackage{pgf}

\providecommand{\Pihdiv}{\Pi_{h,\divergence}}
\providecommand{\tPihdiv}{\overline{\Pi}_{h,\divergence}}

\providecommand{\dx}{\,\mathrm{d}x}
\providecommand{\ds}{\,\mathrm{d}s}
\providecommand{\tria}{\mathcal{T}}

\providecommand{\nodes}{\mathcal{N}}
\providecommand{\edges}{\mathcal{E}}
\providecommand{\faces}{\mathcal{F}}

\providecommand{\edgesint}{\mathcal{E}^\circ}
\providecommand{\edgesbnd}{\mathcal{E}^\partial}
\providecommand{\nodesint}{\mathcal{N}^\circ}

\providecommand{\edgespace}{V_{\textup{tan}}} 
\providecommand{\Qaug}{Q_h^\textup{+}}

\makeatletter
\@namedef{subjclassname@2020}{%
  \textup{2020} Mathematics Subject Classification}
\makeatother

\begin{document}

\author[L.\ Diening]{Lars Diening}
\author[J.\ Storn]{Johannes Storn}
\author[T.\ Tscherpel]{Tabea Tscherpel}

\address[L.\ Diening, J.\ Storn, T.\ Tscherpel]{Department of Mathematics, University of Bielefeld, Postfach 10 01 31, 33501 Bielefeld, Germany}
\email{lars.diening@uni-bielefeld.de}
\email{jstorn@math.uni-bielefeld.de}
\email{ttscherpel@math.uni-bielefeld.de}
\thanks{This research was supported by the DFG through the CRC 1283 ``Taming uncertainty and profiting from randomness and low regularity in analysis, stochastics and their applications"}


\title{Fortin Operator for the Taylor--Hood Element}

\begin{abstract}
  We design a Fortin operator for the lowest-order Taylor--Hood element in any dimension, which was previously constructed only in 2D. 
  In the construction we use tangential edge bubble functions for the divergence correcting operator. 
  This naturally leads to an alternative inf-sup stable reduced finite element pair. 
Furthermore, we provide a counterexample to the inf-sup stability and hence to existence of a Fortin operator for the $P_2$--$P_0$ and the augmented Taylor--Hood element in 3D.
\end{abstract}

\maketitle

\section{Introduction}
\label{sec:introduction}

Inf-sup stable finite element pairs are a necessity in the design of stable numerical schemes for the Stokes equations and related problems. 
A common tool to verify inf-sup stability are Fortin operators, which are bounded interpolation operators preserving the discrete divergence of a function. 
Besides stability results, Fortin operators are important in the design of a posteriori error estimators \cite{LedererMerdonSchoeberl2019}, their quasi-local approximation properties are needed when discretizing nonlinear incompressible fluid equations \cite{GiraultScott2003}, they are used in the investigation of pre-conditioners \cite{MardalSchoeberlWinther2013}, and they allow for stability results in different norms such as $W^{1,\infty}$ \cite{GNS.2005,GuzmanSanchez2015}. 
Hence, there are numerous contributions on the design of Fortin operators for various finite element pairs, including several papers \cite{BW.2020,Chen2014,Falk2008,GiraultScott2003,MardalSchoeberlWinther2013} on the Taylor--Hood element in dimension $d=2$. 

The lowest-order Taylor--Hood element uses continuous piecewise quadratic functions for the velocity space and of continuous, piecewise affine functions for the pressure space. 
This pair is of particular interest, since it is the lowest-order conforming stable element that ensures the same approximation order for velocity and pressure functions. 
Furthermore, a sequence of nested spaces is formed when refining the mesh, which is advantageous in the numerical analysis of adaptive schemes, cf.~\cite{Feischl19}.
While inf-sup stability is known to hold in dimension $d=3$ for the lowest-order Taylor--Hood element \cite{Bo.1997}, the construction of a Fortin operator is still an open problem. 
Closing this gap for all dimensions $d\geq 2$ is the main purpose of this paper. 
For higher-order versions of the Taylor--Hood element a Fortin operator is constructed in \cite{GiraultScott2003} for polynomial order $k \geq d$ of the velocity space.

A customary tool in the design of Fortin operators are face bubble functions. 
However, in dimensions $d\geq 3$ face bubble functions are not quadratic and hence are not contained in the discrete velocity space. This causes difficulties in the construction. 
A partial remedy are non-constructive approaches via (local) discrete inf-sup conditions \cite{ErnGuermond2004,GuzmanSanchez2015}. 
However, those approaches do not allow for certain beneficial properties achieved by the constructive design such as local $W^{1,p}$-stability for all $p \in [1,\infty]$.
As emphasized in \cite[p.~599]{GuzmanSanchez2015} and \cite[p.~53]{JKN.2018} a constructive design of a Fortin operator for the lowest-order Taylor--Hood element for $d\geq 3$ is still an open problem. 
In this paper we solve this open problem.

As usual we combine a divergence correcting operator with an interpolation operator. 
We overcome the need for face bubble functions by use of tangential edge bubble functions. 
The latter have previously been used in~\cite{MardalSchoeberlWinther2013} for the construction of a Fortin operator in 2D. 
Therein the authors exploit a correspondence between the tangential edge bubble functions and a basis of the lowest order N\'{e}d\'{e}lec elements of the first kind. 
Here we directly work with tangential edge bubble functions allowing for a construction in general dimensions. 

Our Fortin operator is locally $W^{1,p}$-stable and can be modified to obtain a locally $L^p$-stable version. 
Such a modification is of particular interest for a singularly perturbed Stokes problem and was investigated for 2D in~\cite{MardalSchoeberlWinther2013}.

The construction of our divergence correcting operator naturally leads to a reduced Taylor--Hood finite element pair, for which the velocity space is spanned by piecewise affine functions and tangential edge bubble functions. 
Our Fortin operator adapts to this reduced finite element pair and hence inf-sup stability is guaranteed. 
In 3D the dimension of the finite element pair is significantly smaller than the one of the MINI element developed in~\cite{ABF.1984}. 
For a standard uniform simplicial partition the dimension of the discrete function space is halved. 
To the best of our knowledge this reduced pair is known only in the 2D case,  cf.~\cite{MardalSchoeberlWinther2013}.  

A further alternative finite element pair is the augmented (sometimes called enriched or modified) Taylor--Hood element. 
The function space pair results from the Taylor--Hood element by adding piecewise constant functions to the pressure space. 
In 2D this enriched pair of discrete function spaces is still inf-sup stable, cf.~\cite{T.1990}. 
The same is true for higher-order augmented Taylor--Hood elements and higher dimension, if $k \geq d$ \cite{BCGG.2012}. 
However, numerical experiments in 3D show a lack of stability for the lowest-order version, see for example~\cite[Sec.~3]{GWW.2014}. 
We present a simple explicit example that confirms the experimental evidence. 

We construct the Fortin operator for the lowest-order Taylor--Hood element in Section~\ref{sec:diverg-pres-fort}. 
More specifically, Section~\ref{subsec:GeoSetup} contains some notation needed throughout this paper. 
In Section~\ref{subsec:TangBubFct} we introduce and investigate tangential edge bubble functions. 
Those bubble functions are utilized in Section~\ref{sec:Pi2} to design a divergence correcting operator. 
In Section~\ref{subsec:Fortin} we employ the latter to construct our Fortin operator. 
An $L^p$-stable version is discussed in Section~\ref{sec:LpFortin}.

We conclude with an investigation of alternative finite element pairs in Section \ref{sec:alternative}. 
In particular, for any dimension $d\geq 2$ we introduce and briefly discuss a reduced Taylor--Hood element in Section \ref{sec:reduced}. The augmented lowest-order Taylor--Hood element and the $P_2$--$P_0$ element are considered in Section \ref{sec:enriched-taylor-hood}.   

\section{Fortin operator}
\label{sec:diverg-pres-fort}

In this section we construct the divergence preserving Fortin operator for the lowest-order Taylor--Hood element for any dimension~$d \geq 2$.

\subsection{Geometric setup and notation}\label{subsec:GeoSetup}

Throughout this paper, let the domain $\Omega \subset \mathbb{R}^d$ be an open (bounded) polytope with underlying regular partition $\tria$ into closed $d$-simplices. 
We denote by $\nodes$ and $\edges$ the set of nodes and edges in $\tria$, respectively. 
Further, let $\nodes^\circ$ and $\edges^\circ$ denote the subsets of interior nodes and edges, and let $\nodes^\partial$ and $\edges^\partial$ be the subsets of boundary nodes and edges, respectively. 
We call the $(d-1)$-facets of simplices in~$\tria$ \emph{faces}. 
The set of all faces in $\tria$ is denoted by~$\faces$. 
For points~$a_1,\dots, a_m \in \Rd$ we denote by $[a_1,\dots,a_m]\subset \mathbb{R}^d$ the convex hull of~$\set{a_1,\dots, a_m}$. 
This allows us to represent (undirected) edges and faces by its nodes. 
The local mesh size is defined by $h_T \coloneqq \diameter(T)$ for all $T\in \tria$. 
Furthermore, the mesh size function is given by $h \coloneqq \sum_{T \in \tria} h_T \indicator_T$, where $\indicator_T$ is the indicator function of~$T$. 
We denote the Lebesgue measure of a set~$U \subset \Rd$ by $\abs{U}$. For a set $U \subset \Rd$ with $\abs{U} > 0$ we define $\dashint_U f \dx \coloneqq \abs{U}^{-1}\int_U f \dx$ as the integral mean of a function~$f$ over~$U$.

We require the following standard assumption on the simplices at the boundary, cf.~\cite[Thm.~8.8.2]{BBF.2013}.
\begin{assumption}
  \label{ass:boundary}
  Each $d$-simplex $T\in \tria$ has at least one interior node
  $i \in \nodes^\circ$.
\end{assumption}

For $1 \leq p \leq \infty$ let $W^{1,p}(\Omega)$ and $W^{1,p}(\Omega;\Rd)$ denote the standard Sobolev space of functions mapping to~$\mathbb{R}$ and $\Rd$, respectively. 
The corresponding notation shall be used for other function spaces of vector-valued functions. 
Let $W^{1,p}_0(\Omega)$ denote the subspace of functions with zero trace on the boundary $\partial \Omega$ of $\Omega$. 
Furthermore, for $1 \leq p \leq \infty$ let $L^p(\Omega)$ be the standard Lebesgue space. 
We denote by $L^p_0(\Omega)$ the subspace of functions $q$ with vanishing integral $\int_\Omega q \dx = 0$ and write $\skp{\bigcdot}{\bigcdot}$ for the $L^2(\Omega)$-scalar product.

Let $\mathcal{L}^s_k(\tria)$ for $s,k \in \setN_0$ be the Lagrange
space of functions in~$W^{s,1}(\Omega)$ that are piecewise
polynomials of order at most~$k$. 
Let $\phi_i$ with $i\in \nodes$ denote the standard nodal basis of~$\mathcal{L}^1_1(\tria)$ forming a partition of unity. 
Let $\omega_0(i)$ denote the support of~$\phi_i$ that coincides with the closed nodal patch of~$i \in \nodes$.

The lowest-order Taylor--Hood element uses quadratic Lagrange elements for the velocity and linear Lagrange elements for the
pressure, i.e., 
\begin{align*}
  V_h \coloneqq \mathcal{L}^1_2(\tria; \Rd) \cap W^{1,1}_0(\Omega;\Rd)\qquad\text{and}\qquad
  Q_h \coloneqq \mathcal{L}^1_1(\tria) \cap L^2_0(\Omega).
\end{align*}
As usual we employ a projection~$\Pi_1$ with approximation properties and a divergence correcting linear operator~$\Pi_2$ to construct the Fortin operator $\Pihdiv$ as
\begin{align*}
  \Pihdiv \coloneqq \Pi_1 + \Pi_2(\identity - \Pi_1).
\end{align*}
For $\Pi_1$ one may choose the standard Scott--Zhang operator~\cite{SZ.1990} applied componentwise. 
The challenge lies in constructing~$\Pi_2$ in the absence of face bubble functions. 
We introduce this operator in Section~\ref{sec:Pi2}. 
It is based on tangential edge bubble functions investigated in the following Section~\ref{subsec:TangBubFct}. 
In Section~\ref{subsec:Fortin} we collect the properties of the resulting Fortin operator~$\Pihdiv$ for the simplest choice of $\Pi_1$.

\subsection{Tangential bubble functions}\label{subsec:TangBubFct}
The main tool in our design of a divergence correcting
operator~$\Pi_2$ are tangential edge bubble functions studied in
this section.

Given an edge $[i,j] \in \edges$ with adjacent nodes $i,j\in \nodes$ and nodal basis functions $\phi_i,\phi_j \in \mathcal{L}_1^1(\tria)$, the (directed) \emph{tangential bubble function} is given by $\varphi_i\varphi_j(j-i)$. 
Let $\omega_{i,j} \coloneqq \textup{supp}(\varphi_i\varphi_j)$ denote the closed edge patch, and with slight abuse of notation also the set of simplices $\omega_{i,j} \coloneqq \lbrace T \in \tria\colon [i,j]\subset T\rbrace$.
\begin{lemma}[Tangential bubble]\label{lem:bubble-orth}
  For any edge $[i,j] \in \mathcal{E}$ and node $k \in \mathcal{N}$ we have
   \begin{align*}
     \skp{\divergence (\phi_i \phi_j(j-i))}{\phi_k} =  
\begin{cases}     
     0 \quad &\text{ if }  k \notin \lbrace i,j\rbrace,\\
\abs{\omega_{i,j}} \frac{d!}{(2+d)!} \quad &\text{ if }  k = i,\\
- \abs{\omega_{i,j}} \frac{d!}{(2+d)!} \quad &\text{ if }  k = j. 
\end{cases}
   \end{align*}
\end{lemma}
\begin{proof}
  If the node $k \notin \omega_{i,j}$ is not contained in the closed edge patch $\omega_{i,j}$, then the supports of $\phi_k$ and of $\phi_i\phi_j$ do not intersect and hence the expression vanishes. 
If the node $k \in \omega_{i,j}$ is contained in the edge patch  integrating by parts yields
  \begin{align*}
    \skp{\divergence (\phi_i \phi_j(j-i))}{\phi_k}
    &= - \int\limits_{\omega_{i,j}} \phi_i \phi_j (j-i) \cdot \nabla \phi_k \dx +\!\!\!\! \int\limits_{\partial \omega_{i,j}} \!\!\!\!\phi_i \phi_j \phi_k (j-i) \cdot \nu \ds,
  \end{align*}
  where $\nu$ is the outer unit normal vector on $\partial \omega_{i,j}$. 
Since we have $(j-i) \cdot \nu|_f = 0$ on all faces $f\in \mathcal{F}$ with $f \subset \partial \Omega$ and $[i,j] \subset f$ as well as $\phi_i\phi_j|_f = 0$ on all other faces $f\in \mathcal{F}$ with $f\subset \partial \omega_{i,j}$, the boundary integral is zero.  
Hence, we obtain that
  \begin{align}
    \label{eq:aux1}
    \skp{\divergence (\phi_i \phi_j(j-i))}{\phi_k} 
    &= - \sum_{T \in \omega_{i,j}} \int_{T} \phi_i \phi_j (j-i) \cdot \nabla \phi_k \dx.
  \end{align}
  If the node $k \notin \lbrace i,j\rbrace$, we have that $\phi_k|_{[i,j]}=0$. 
  Therefore, the piecewise constant function $\nabla \phi_k \in \mathcal{L}^0_0(\tria;\mathbb{R}^d)$ is orthogonal to the edge vector~$(j-i)$ on each~$T \in \omega_{i,j}$. 
  Thus, for each node $k \notin[i,j]$ the integrals vanishes and the claim follows.

  If $k = i$, we use the identity $\varphi_i(x) = 1 + \nabla \varphi_i|_T \cdot (x-i)$ for all $x\in T$ and $T \in \omega_{i,j}$ to conclude that $\nabla \varphi_i|_T \cdot (j-i) = -1$. 
  Applying this in~\eqref{eq:aux1} yields
  \begin{align*}
    \skp{\divergence (\phi_i \phi_j(j-i))}{\phi_k}
    &=  \sum_{T\in \omega_{i,j}}\int_T \phi_i \phi_j \dx
    =  \abs{\omega_{i,j}} \frac{d!}{(2+d)!}. 
  \end{align*}
  Exchanging the roles of $i$ and $j$ shows the claim for $k=j$ and finishes the proof. 
\end{proof}
Lemma~\ref{lem:bubble-orth} motivates for any edge $[i,j] \in \edges$ the definition of the \emph{normalized tangential edge bubble function}
\begin{align}\label{def:psi}
  b_{i,j} &\coloneqq \frac{(2+d)!}{d! \abs{\omega_{i,j}}} \phi_i \phi_j(j-i). 
\end{align}
These functions satisfy for any edge $[i,j] \in \edges$ and any node $k\in \nodes$ the identity
\begin{align}
  \label{eq:orth-bubble1}
    \skp{\divergence b_{i,j}}{\phi_k} 
    = \delta_{i,k} - \delta_{j,k}.
\end{align}
If $[i,j] \in \edgesint$ is an interior edge, then the function $b_{i,j}$ is contained in~$V_h$. 
However, if the edge $[i,j] \in \edgesbnd$ is on the boundary $\partial \Omega$, then the function $b_{i,j}$ does not vanish on~$\partial \Omega$ and is therefore not an element of~$V_h$. 
Thus, it cannot be used in the divergence correction. 
For this reason in the following we replace it using tangential bubble functions associated to adjacent interior edges.
By Assumption~\ref{ass:boundary} for each $[i,j] \in \edgesbnd$ there exists an interior node~$m \in \mathcal{N}^\circ$ such that
$[i,m]\in \edgesint$ and $[j,m]\in \edgesint$.  
For any edge $[i,j]\in \edges$ we define the \emph{modified tangential bubble function} $\psi_{i,j}$ by
\begin{align*}
  \psi_{i,j} \coloneqq
               \begin{cases}
                 b_{i,j} & \text{ if }[i,j] \in \edgesint,
                 \\
                 b_{i,m} + b_{m,j} & \text{ if } [i,j] \in \edgesbnd \text{ for chosen } m \in \mathcal{N}^\circ \text{ with }[i,m],[j,m] \in \edgesint.
               \end{cases}
\end{align*}
As for the tangential bubble functions $b_{i,j}$ we have $\psi_{j,i} = - \psi_{i,j}$.
Since only interior tangential bubble functions are used each $\psi_{i,j}$ vanishes on~$\partial \Omega$ and consequently one has that $\psi_{i,j} \in V_h$ for all edges $[i,j] \in \edges$.
Note that for each interior edge $[i,j] \in \edgesint$ the function $\psi_{i,j}$ is supported on the edge patch $\omega_{i,j}$. 
In contrast, for boundary edges $[i,j] \in \edgesbnd$ the support of $\psi_{i,j}$ given by $\support(\psi_{i,j}) =  \omega_{i,m} \cup \omega_{j,m}$ is larger. 
The identity in~\eqref{eq:orth-bubble1} also holds
for~$\psi_{i,j}$, that is
\begin{align}
  \label{eq:orth-bubble2}
    \skp{\divergence \psi_{i,j}}{\phi_k} 
    = \delta_{i,k} - \delta_{j,k}
     \qquad \text{for any } [i,j] \in \edges \text{ and any }k \in \nodes.
\end{align}
We denote the space spanned by interior tangential bubble functions by
\begin{align}\label{eq:edgespace}
  \edgespace \coloneqq
  \linearspan
  \{\phi_i \phi_j (j-i) \colon [i,j] \in \edgesint\} \subset V_h.
\end{align}
By definition we have that $\psi_{i,j} \in \edgespace$ for any edge $[i,j] \in \edges$.

\subsection{\texorpdfstring{Divergence correcting operator~$\Pi_2$}{Divergence correcting operator}}
\label{sec:Pi2}

In this section we construct the new divergence correcting
operator~$\Pi_2$ based on interior tangential bubble functions.
For each node $i \in \nodes$ we define the operator
$\Pi_{2,i} \colon W^{1,1}_0(\Omega;\Rd) \to \edgespace$ by
\begin{align}
  \label{eq:Pi2i}
  \Pi_{2,i}v &\coloneqq  \sum_{j\in \nodes \colon [i,j] \in \edges} \skp{\divergence(\phi_i v)}{\phi_j} \psi_{j,i} =
   -\!\!\!\!\! \sum_{j \in \nodes \colon [i,j] \in \edges}  \skp{\phi_i v}{ \nabla \phi_j} \psi_{j,i}.
\end{align}
The second identity allows us to extend~$\Pi_{2,i}$ to an operator $\Pi_{2,i} \colon L^1(\Omega;\Rd) \to \edgespace$.
The partition of unity $1 = \sum_{\ell\in \mathcal{N}} \phi_\ell$ and application of the identity $\langle \divergence( \phi_iv),1\rangle = 0$ for $i \in \nodes$ yield for all $i,k \in \mathcal{N}$ and $v\in W^{1,1}_0(\Omega;\Rd)$ that 
\begin{align}
  \label{eq:div-Pi2i}
  \begin{aligned}
    \skp{\divergence(\Pi_{2,i} v)}{\phi_k} &= \sum_{j\in \nodes \colon [i,j] \in \edges} \skp{\divergence(\phi_i v)}{\phi_j}\, \skp{\divergence \psi_{j,i}}{\phi_k}
    \\
    &= \sum_{j \in \nodes \colon [i,j] \in \edges} \skp{\divergence(\phi_i v)}{\phi_j}\, (\delta_{j,k}-\delta_{i,k})
    \\
    &=
    \begin{cases}
      \skp{\divergence(\phi_i v)}{\phi_k} \quad& \text{for } k \neq i;
      \\
      \skp{\divergence(\phi_i v)}{-\sum_{j \in \mathcal{N} \setminus\set{k}} \phi_j}  \quad & \text{for }k = i
    \end{cases}
  \\
  &=
    \begin{cases}
      \skp{\divergence(\phi_i v)}{\phi_k} & \text{for }k \neq i,
      \\
      \skp{\divergence(\phi_i v)}{\phi_k-1} \quad \qquad\quad \;\;\; & \text{for }k = i
    \end{cases}
  \\
  &= \skp{\divergence(\phi_i v)}{\phi_k}.
\end{aligned}
\end{align}
We define for any $v \in L^{1}(\Omega;\Rd)$ the global operator
\begin{align}\label{def:Pi2}
  \Pi_2 v \coloneqq \sum_{i \in \mathcal{N}} \Pi_{2,i}v.  
\end{align}
Summing over all edges yields for all $v \in W^{1,1}_0(\Omega;\Rd)$ the identity
\begin{align}
  \label{eq:Pi2o-summary}
  \Pi_2 v &= \sum_{[i,j] \in \edges} \big( \skp{\divergence(\phi_i v)}{\phi_j} - \skp{\divergence(\phi_j v)}{\phi_i}\big) \psi_{j,i}.
\end{align}
For functions $v \in L^1(\Omega;\Rd)$ the operator is defined via the equivalent formulation
\begin{align}
  \label{eq:Pi2o-summaryL1}
  \Pi_2 v
  &= \sum_{[i,j] \in \edges} \skp{v}{-\phi_i \nabla \phi_j+ \phi_j \nabla \phi_i}\, \psi_{j,i}.
\end{align}
By \eqref{eq:div-Pi2i} and the partition of unity $1 = \sum_{\ell \in \nodes} \phi_\ell$ we find for all $v \in W^{1,1}_0(\Omega;\Rd)$ that
\begin{align}
  \label{eq:P2-div}
  \skp{\divergence \Pi_2v}{\phi_k} 
  &= \sum_{i \in \mathcal{N}}\skp{\divergence\Pi_{2,i}v}{\phi_k}  
    = \sum_{i \in \mathcal{N}}\skp{\divergence(\phi_i v)}{\phi_k}  = \skp{\divergence  v}{\phi_k}. 
\end{align}
For any $K \subset \overline{\Omega}$ we define the closed nodal and edge patch/neighborhood by
\begin{align*}
  \omega_0(K) & \coloneqq \bigcup \lbrace T' \in \tria \colon  \text{there exists } j\in \nodes \text{ with } j\in K \cap T' \rbrace,
  \\
  \omega_1(K) & \coloneqq \bigcup \lbrace T' \in \tria \colon \text{there exists } e\in \edges\text{ with }e\subset K \cap T'\rbrace.
\end{align*}
Note that $\omega_0(i) = \omega_0(\set{i})$ and $\omega_{i,j} = \omega_1([i,j])$ for $i,j \in \nodes$ and $[i,j] \in \edges$.
\begin{proposition}[Properties of $\Pi_2$]
  \label{pro:Pi2-div}
  The operator $\Pi_2 \colon L^1(\Omega;\Rd) \to \edgespace$ satisfies the following properties.
  \begin{enumerate}
  \item \label{itm:Pi2-div}
    (Divergence) We have for all nodes $k \in \mathcal{N}$ and any $v \in W^{1,1}_0(\Omega;\Rd)$ that
    \begin{align*}
      \skp{\divergence(\Pi_2v)}{\phi_k}= \skp{\divergence{v}}{\phi_k}.
    \end{align*}  
  \item \label{itm:Pi2-stab} (Local stability) 
    One has for all $v \in L^1(\Omega;\Rd)$ and any $T\in \tria$ that
    \begin{align*}
      \norm{\Pi_2v}_{L^{\infty}(T)} \lesssim \dashint_{\omega_0(T)} \abs{v} \dx.
    \end{align*}
  \end{enumerate}
      If $T \cap \partial \Omega = \emptyset$ the estimate in \ref{itm:Pi2-stab} holds with $\omega_0(T)$ replaced by the smaller set $\omega_1(T)$.  The hidden constant depends only on the dimension $d$ and the shape regularity of~$\tria$.
\end{proposition}
\begin{proof}
  Since~\ref{itm:Pi2-div} follows by \eqref{eq:P2-div}, it remains to prove~\ref{itm:Pi2-stab}. 
  For $T \in \tria$ we estimate
  \begin{align}\label{eq:addends}
    \norm{\Pi_2v}_{L^{\infty}(T)}
    &\leq \sum_{[i,j] \in \edges} 
      \bigabs{ \skp{v}{-\phi_i \nabla \phi_j+ \phi_j \nabla \phi_i}}\, \norm{\psi_{i,j}}_{L^\infty(T)}.
  \end{align}
  Note that $\norm{\psi_{i,j}}_{L^\infty(T)}$ is zero unless
  $T \subset  \support(\psi_{i,j})$. 
  Hence, the number of terms in the sum is
  bounded by a constant only depending on the shape regularity. 
  
  Suppose first that $[i,j] \in \edgesint$ is an interior edge with $T\subset \support(\psi_{i,j})$. 
  Then we have that $T \subset \omega_{i,j}$. 
  By the definition of the tangential bubble functions we have
  \begin{align*}
    \norm{b_{i,j}}_{L^{\infty}(\Omega)} \lesssim \frac{1}{\abs{\omega_{i,j}}}\abs{j-i} \lesssim \frac{h_T}{\abs{\omega_{i,j}}}.
  \end{align*}
For $\psi_{i,j}$ the analogous estimate holds true. 
  Since $T\subset  \support(-\phi_i \nabla \phi_j+ \phi_j \nabla \phi_i) =\omega_{i,j} \subset \omega_1(T)$  the corresponding term in \eqref{eq:addends} may be estimated as
  \begin{align*}
    \bigabs{ \skp{v}{-\phi_i \nabla \phi_j+ \phi_j \nabla \phi_i}}\, \norm{\psi_{i,j}}_{L^\infty(T) } &\lesssim \int_{\omega_{i,j}} \abs{v}\, h_T^{-1}\dx\, \frac{h_T}{\abs{\omega_{i,j}}}
                                                                                                        \lesssim \dashint_{\omega_1(T)} \abs{v}\dx. 
  \end{align*}
  Suppose now that $[i,j] \in \edgesbnd$ is an edge on the boundary with $T\subset \support(\psi_{i,j})$. 
    Recall that there is an interior node $m \in \nodesint$ with  $\psi_{i,j} = b_{i,m} + b_{m,j}$ and thus it follows that  $T \in \omega_{i,m} \cup \omega_{m,j}$. 
    Consequently, $T$ contains the node $i$ or $j$ and hence it follows that $\omega_{i,j} \subset \omega_0(T)$. 
Therefore, the respective term in \eqref{eq:addends} is bounded by
  \begin{align*}
    \bigabs{ \skp{v}{-\phi_i \nabla \phi_j+ \phi_j \nabla \phi_i}}\, \norm{\psi_{i,j}}_{L^\infty(T) } &\lesssim \int_{\omega_{i,j}} \abs{v}\, h_T^{-1}\dx\, \frac{h_T}{\abs{\omega_{i,j}}}
                                                                                                        \lesssim \dashint_{\omega_0(T)} \abs{v}\dx.
  \end{align*}
  Combining both cases and the inclusion $\omega_1(T) \subset \omega_0(T)$ prove~\ref{itm:Pi2-stab}.
\end{proof}

\subsection{Fortin operator}
\label{subsec:Fortin}

Let $\Pi_1\colon W^{1,1}(\Omega;\Rd) \to \mathcal{L}^1_2(\tria;\Rd)$ denote the standard Scott--Zhang operator~\cite{SZ.1990} which preserves discrete traces applied componentwise. 
We define the Fortin operator $\Pihdiv$ by
\begin{align}
  \label{eq:def-Phihdiv}
  \Pihdiv \coloneqq \Pi_1 + \Pi_2(\identity - \Pi_1).
\end{align}
The Scott--Zhang operator~$\Pi_1$ is a linear projection that maps $W^{1,1}(\Omega;\Rd)$ to $\mathcal{L}^1_2(\tria;\Rd)$ and $W^{1,1}_0(\Omega;\Rd)$ to $V_h$ and satisfies for all $T\in \tria$ the local $W^{1,1}$-stability
\begin{align}
  \label{eq:SZ-W11-stab}
  \norm{\Pi_1 v}_{L^{\infty}(T)}
  &\lesssim
    \dashint_{\omega_0(T)} \abs{v} \dx + 
    h_T \dashint_{\omega_0(T)} \abs{\nabla v} \dx\quad\text{for all }v\in W^{1,1}(\Omega;\mathbb{R}^d).
\end{align}
The hidden constant depends only on $d$ and the shape regularity of $\tria$.

\begin{proposition}[Fortin operator]
  \label{pro:fortin-basic}
  The operator $\Pihdiv$ is a linear projection from $W^{1,1}(\Omega;\Rd)$ to $\mathcal{L}^1_2(\tria;\Rd)$ and maps $W^{1,1}_0(\Omega;\Rd)$ to $V_h$. It preserves discrete traces and satisfies the following additional properties.
  \begin{enumerate}
  \item \label{itm:fortin-basic-div}
    \emph{(Divergence preserving)}
    We have for all $v \in W^{1,1}_0(\Omega;\Rd)$ and all $q_h \in Q_h$ that
    \begin{align*}
      \skp{\divergence\,\Pihdiv v}{q_h}= \skp{\divergence{v}}{q_h}.
    \end{align*}  
  \item
    \label{itm:fortin-basic-stab}
    \emph{(Local $W^{1,1}$-stability)}
    We have for all $v \in W^{1,1}(\Omega;\Rd)$ 
    \begin{align*}
      \norm{\Pihdiv v}_{L^{\infty}(T)}
      &\lesssim
        \dashint_{\omega_0(\omega_0(T))} \abs{v} \dx + 
        h_T \dashint_{\omega_0(\omega_0(T))} \abs{\nabla v} \dx.
    \end{align*}
  \end{enumerate}
  If $T \cap \partial \Omega = \emptyset$ the estimate in \ref{itm:fortin-basic-stab} holds with $\omega_0(\omega_0(T))$ replaced by the smaller set $\omega_0(\omega_1(T))$. 
  The hidden constant depends only on $d$ and the shape regularity of $\tria$. 
\end{proposition}
\begin{proof}
  Note that the linear operator $\Pi_1$ projects $W^{1,1}(\Omega;\Rd)$ to $\mathcal{L}^1_2(\tria;\Rd)$ and $W^{1,1}_0(\Omega;\Rd)$ to $V_h$ and it preserves discrete traces. 
  Moreover, the linear operator $\Pi_2$ maps $L^1(\Omega;\Rd)$ to~$V_h$. 
  Thus, the operator $\Pihdiv \coloneqq \Pi_1 + \Pi_2(\identity - \Pi_1)$ is a linear projection from $W^{1,1}(\Omega;\Rd)$ to $\mathcal{L}^1_2(\tria;\Rd)$ and from $W^{1,1}_0(\Omega;\Rd)$ to $V_h$ that preserves discrete traces. 

   For all $v \in W^{1,1}_0(\Omega;\Rd)$ by Proposition~\ref{pro:Pi2-div}\ref{itm:Pi2-div} the operator $\Pi_2$ preserves the discrete divergence and we have that $v - \Pi_1 v \in W^{1,1}_0(\Omega;\Rd)$. 
   Hence, for any $v \in W^{1,1}_0(\Omega;\Rd)$ the following identity holds in $Q_h^*$
  \begin{align*}
    \divergence \Pihdiv v
    &= \divergence \Pi_1 v + \divergence\big(\Pi_2(v - \Pi_1 v)\big)
    = \divergence \Pi_1 v + \divergence(v - \Pi_1 v) = \divergence v,
  \end{align*}
  which proves~\ref{itm:fortin-basic-div}. 

  Combining Proposition~\ref{pro:Pi2-div}\ref{itm:Pi2-stab} and~\eqref{eq:SZ-W11-stab} results for all $v \in W^{1,1}(\Omega;\Rd)$ in
  \begin{align*}
    \norm{\Pihdiv v}_{L^{\infty}(T)}
    &\leq
      \norm{\Pi_1 v}_{L^{\infty}(T)} + \norm{\Pi_2( v - \Pi_1 v)}_{L^{\infty}(T)}
    \\
    &\leq
      \norm{\Pi_1 v}_{L^{\infty}(T)} + \dashint_{\omega_0(T)} \abs{v - \Pi_1 v}\,dx
    \\
    &\lesssim 
      \dashint_{\omega_0(\omega_0(T))} \abs{ v} \dx  + h_T
      \dashint_{\omega_0(\omega_0(T))} \abs{\nabla v} \dx.
  \end{align*}
  If $T$ is an interior simplex we have the domain $\omega_0(\omega_1(T))$ on the right-hand side of the estimate due to the smaller domain of dependence of $\Pi_2$, see Proposition~\ref{pro:Pi2-div}. 
\end{proof}
From the basic properties of $\Pihdiv$ by standard arguments we can derive $W^{1,p}$-stability and approximation properties. 
\begin{proposition}[Fortin operator]
  \label{pro:fortin-stab-approx}
  One has the following estimates for any $v \in W^{1,p}(\Omega;\Rd)$
  with $p \in [1,\infty]$ and all~$T \in \tria$, where the hidden constants depend only on $d$ and the shape regularity of $\tria$. 
  \begin{enumerate}
  \item \emph{(Local $W^{1,p}$-stability)} One has that \label{itm:fortin-stab} %
    \begin{align*}
      \lefteqn{\norm{\Pihdiv v }_{L^p(T)} + h_T
      \norm{\nabla \Pihdiv v }_{L^p(T)}} 
      \qquad      & \\
      &\lesssim \norm{v}_{L^p(\omega_0(\omega_0(T)))}
        + h_T\,\norm{\nabla v}_{L^p(\omega_0(\omega_0(T)))}.
    \end{align*}
  \item \emph{(Approximation)} \label{itm:fortin-approx} If additionally $v \in W^{s,p}(\Omega;\Rd)$ with $s \in \set{1,2,3}$, then we have that
    \begin{align*}
      \norm{ v-\Pihdiv v }_{L^p(T)}  +  h_T \norm{\nabla (v-\Pihdiv v) }_{L^p(T)} &\lesssim h_T^s \norm{\nabla^{s} v}_{L^p(\omega_0(\omega_0(T)))}.
    \end{align*}
  \item  \emph{(Continuity)} \label{itm:fortin-cont} 
  One has that
    \begin{align*}
      \norm{\nabla \Pihdiv v }_{L^p(T)} &\lesssim \norm{\nabla v}_{L^p(\omega_0(\omega_0(T)))}.
    \end{align*}
  \end{enumerate}
      If $T \cap \partial \Omega = \emptyset$, then the estimates hold with $\omega_0(\omega_0(T))$ replaced by $\omega_0(\omega_1(T))$. 
\end{proposition}
\begin{proof}
  The first property~\ref{itm:fortin-stab} follows directly from Proposition~\ref{pro:fortin-basic}\ref{itm:fortin-basic-stab} applying inverse estimates, H\"older's inequality and $\abs{T} \eqsim \abs{\omega_0(\omega_0(T))}$ ensured by shape regularity 
  \begin{align*}
    \lefteqn{\norm{\Pihdiv v }_{L^p(T)} + h_T
    \norm{\nabla \Pihdiv v }_{L^p(T)}} \qquad
    &
    \\
    &\lesssim \norm{1}_{L^p(T)} 
      \norm{\Pihdiv v}_{L^{\infty}(T)}
      \\
    &\lesssim \norm{1}_{L^p(T)}  \Bigg(
      \dashint_{\omega_0(\omega_0(T))} \abs{v} \dx + 
      h_T \dashint_{\omega_0(\omega_0(T))} \abs{\nabla v} \dx \Bigg).
  \end{align*}
  To prove~\ref{itm:fortin-approx} let $g \in \mathcal{L}^1_2(\tria;\Rd)$ by arbitrary. 
  Then, thanks to the projection property we have that $\Pihdiv g = g$. Adding and subtracting $g$ and using~\ref{itm:fortin-stab} yields
  \begin{align*}
    \lefteqn{\norm{ v-\Pihdiv v }_{L^p(T)}  + h_T
    \norm{ \nabla (v-\Pihdiv v) }_{L^p(T)} } \qquad
    &
    \\
    &\lesssim
      \norm{ v- g }_{L^p(T)}  + 
      h_T \norm{ \nabla(v- g) }_{L^p(T)}  
    \\
    &\quad
      + \norm{\Pihdiv( g- v )}_{L^p(T)}
      + h_T
      \norm{ \nabla \Pihdiv( g- v )}_{L^p(T)}
    \\
    &\lesssim 
      \norm{ v- g }_{L^p(\omega_0(\omega_0(T))}  + h_T
      \norm{ \nabla (v- g) }_{L^p(\omega_0(\omega_0(T))}.
  \end{align*}  
  Choosing~$g$ as the averaged Taylor polynomial of order~$s$ (or the best approximating polynomial of order~$s-1$) we obtain by the Bramble--Hilbert lemma in the version of~\cite[Thm.~(4.3.8)]{BreSco2008} that
  \begin{align*}
      \norm{ v-\Pihdiv v }_{L^p(T)}  +  h_T \norm{\nabla (v-\Pihdiv v) }_{L^p(T)}
    &\lesssim h_T^s\,\norm{\nabla^s v}_{L^p(\omega_0(\omega_0(T)))}.
  \end{align*}
  Even if $\omega_0(\omega_0(T))$ is not star-shaped with respect to a ball it is the finite union of such domains. 
  We refer to~\cite{DupSco1980} for the Bramble--Hilbert lemma in this situation. 
  It is also possible to work in the class of John domains and use \Poincare{}'s inequality. 
  See for example~\cite{DieRuzSch10} for \Poincare{}'s inequality on John domains. 
  The John constants of~$\omega_0(\omega_0(T))$ only depend on the shape regularity of~$\tria$. This proves~\ref{itm:fortin-approx}.

  Applying~\ref{itm:fortin-approx} with~$s=1$ shows that
  \begin{align*}
    \norm{\nabla \Pihdiv v }_{L^p(T)} &\leq
    \norm{\nabla v }_{L^p(T)} +
    \norm{\nabla (v-\Pihdiv v) }_{L^p(T)}
    \lesssim \norm{\nabla v}_{L^p(\omega_0(\omega_0(T)))}.
  \end{align*}
  This proves~\ref{itm:fortin-cont}. 
  
  Due to the improved estimate in Proposition~\ref{pro:fortin-basic}
    all estimates hold with the smaller set $\omega_0(\omega_1(T))$ on the right-hand side for interior simplices $T \in \tria$.  
\end{proof}
As usual the local estimates in Proposition~\ref{pro:fortin-stab-approx} imply the corresponding global versions. 
We recall the mesh size function given as $h = \sum_{T \in \tria} \indicator_T h_T$.  
\begin{corollary}
  \label{cor:fortin-global}
  For any $v \in W^{s,p}(\Omega;\Rd)$ with $s \in \set{1,2,3}$ and $p \in [1,\infty]$ we have
  \begin{align*}
    \norm{\Pihdiv v }_{L^p(\Omega)} + 
    \norm{h \nabla \Pihdiv v }_{L^p(\Omega)}
    &\lesssim \norm{v}_{L^p(\Omega)}
      + \norm{h \nabla v}_{L^p(\Omega)},
    \\
    \norm{ v-\Pihdiv v }_{L^p(\Omega)}  +  \norm{h \nabla (v-\Pihdiv v) }_{L^p(\Omega)} &\lesssim  \norm{h^s \nabla^{s} v}_{L^p(\Omega)},
    \\
    \norm{\nabla \Pihdiv v }_{L^p(\Omega)} &\lesssim \norm{\nabla v}_{L^p(\Omega)}.
  \end{align*} 
  The hidden constants depend only on $d$ and the shape regularity of $\tria$.
\end{corollary}
\begin{remark}[Orlicz spaces]
  Corresponding local and global estimates in terms of Orlicz functions can be obtained applying the methods of~\cite{Belenki12}. Such estimates are useful in the context of non-Newtonian fluids.
\end{remark}

\begin{remark}[Domain of dependence]
  It follows from Proposition~\ref{pro:fortin-basic} that the domain of dependence of $(\Pihdiv v)|_T$ is at most $\omega_0(\omega_0(T))$. If $T \cap \partial \Omega = \emptyset$, then it is reduced to~$\omega_0(\omega_1(T))$. 
 In the following we show that it is possible to reduce this domain of dependence for interior simplices even further.

  Suppose that $T \cap \partial \Omega = \emptyset$. 
  Then the domain of dependence of~$\Pi_2$ is $\omega_1(T)$ and the one of $\Pi_1$ is $\omega_0(T)$. 
  The composition~$\Pi_2\Pi_1$ is the crucial term in~\eqref{eq:def-Phihdiv} and leads to the fact that the domain of dependence can be as large as~$\omega_0(\omega_1(T))$. 
  But in the composition $\Pi_2 \Pi_1$ the operator $\Pi_2$ is applied only to discrete functions. 
  Some of these discrete functions are not seen by~$\Pi_2$.  This observation leads to the following improvement. 

  The construction of the Scott--Zhang type operator~$\Pi_1$ requires local basis function. 
  The set of basis functions can be divided into two groups: some associated to nodes and others associated to edges. 
  The coefficients of~$\Pi_1 v$ in this local basis are chosen as certain integral averages over simplices. 
  For a node~$i$ the average is taken over a simplex contained in the larger nodal patch~$\omega_0(i) = \support \phi_i$. 
  For an edge $[i,j]$ the average is taken over a simplex contained in the smaller set~$\omega_{i,j}$. 
  Therefore, it is useful to modify the local basis functions associated to nodes in such a manner that $\Pi_2$ has only a small impact on them. 

  Note that the functions $\phi_i^2$ for $i \in \nodes^\circ$ and $\phi_i \phi_j$ for $[i,j] \in \edgesint$ form a basis of~$\mathcal{L}^2_1(\tria) \cap W^{1,1}_0(\Omega)$. 
  Let us now replace each nodal basis function $\phi_i^2$ by the function
  \begin{align*}
    \rho_i \coloneqq  \phi_i^2 -\frac{2}{d+1}\,\sum_{\lbrace j \in \nodes\colon [i,j] \in \edges\rbrace} \phi_i \phi_j\qquad\text{for all }i \in \nodes^\circ.
  \end{align*}
  These new basis functions are still supported in $\omega_0(i)$ for any $i \in \nodesint$. 
  Then, denoting the $k$-th unit vector by $e_k \in \setR^d$, the vectorial functions $ \phi_i e_k$, for $i \in \nodesint$ and $\phi_i \phi_j e_k$, for $ [i,j]\in \edgesint$, and $k \in \{1, ..., d\}$ form a basis of $V_h$. 
  This choice of basis functions has the advantage that the application of $\Pi_2$ to $\rho_ie_k$ does not extend its support
    for any $i \in \nodesint$, that is $\support(\Pi_2(\rho_i e_k)) \subset \omega_0(i)$.
  Furthermore, for $[i,j] \in \edgesint$ we have $\support(\Pi_2(\phi_i \phi_j) e_k) \subset \omega_1(\omega_{i,j})$.  
  Hence, for interior simplices $T \in \tria $ with $T \cap \partial \Omega = \emptyset$ the domain of dependence of the resulting Fortin operator $(\Pihdiv v)|_T$ is a subset of $\omega_0(T) \cup \omega_1(\omega_1(T))$. 

  All interior local estimates in Proposition~\ref{pro:fortin-basic} and \ref{pro:fortin-stab-approx} then hold with this set on the right-hand side. 
  Note that in dimension $d=2$ we have that $\omega_0(T) \cup \omega_1(\omega_1(T)) = \omega_0(T)$ and thus the Fortin operator is as local as the Scott--Zhang type operator.  
  In dimension $d=3$ the gain compared to $\omega_0(T)$ is not as high because  a simplex contains non-intersecting edges, and hence $\omega_0(T) \subsetneq \omega_1(\omega_1(T))$. 
\end{remark}

\subsection{\texorpdfstring{$L^p$-stable Fortin operator}{Lp-stable Fortin operator}}
\label{sec:LpFortin}

In this section we discuss a modification of the Fortin operator that allows for local $L^p$-stability at the cost of losing preservation of discrete traces. 
Such an operator is employed in the numerical analysis of singularly perturbed Stokes equations. 
More precisely, it ensures uniform estimates with respect to the perturbation parameter $\varepsilon$ in the term $u - \varepsilon \Delta u$. 
In 2D for this purpose an $L^1$-stable Fortin operator was considered  in~\cite{MardalSchoeberlWinther2013}. 
The $L^1$-stability was shown for quasi-uniform and slightly graded mesh. In contrast, our modification applies to any regular mesh without the need for mesh grading conditions. 

Recall that by Proposition~\ref{pro:Pi2-div} the operator $\Pi_2$ is $L^1$-stable. 
Thus, it suffices to adapt the Scott--Zhang type operator $\Pi_1$ used the definition of the Fortin operator in \eqref{eq:def-Phihdiv}. 
Indeed, one only has to replace $\Pi_1$ by a suitable $L^1$-stable version. 
By slight adaptation of the proofs of Propositions~\ref{pro:fortin-basic} and \ref{pro:fortin-stab-approx} we arrive at the following result. 

\begin{proposition}[$L^p$-stable Fortin operator]
  \label{pro:fortin2}
  There exists a linear projection operator $\tPihdiv \colon L^{1}(\Omega;\Rd) \to V_h$ that satisfies the following properties.
  \begin{enumerate}
  \item \label{itm:fortin2-div}
    \emph{(Divergence preserving)}
    For any $v \in W^{1,1}_0(\Omega;\Rd)$ one has that 
    \begin{align*}
      \skp{\divergence\,\tPihdiv v}{q_h}= \skp{\divergence{v}}{q_h} \qquad \text{for all } q_h \in Q_h.
    \end{align*}  
  \item
    \label{itm:fortin2-stab}
    \emph{(Local $L^{p}$-stability)}
    For any $v \in L^{p}(\Omega;\Rd)$ and $p \in [1, \infty]$ we have  that 
    \begin{align*}
      \norm{\tPihdiv v}_{L^{p}(T)}
      &\lesssim
        \norm{v}_{L^p(\omega_0(\omega_0(T)))}.
    \end{align*}
 \item
  \emph{(Approximation)} \label{itm:fortin2-approx} For any $v \in W^{s,p}(\Omega;\Rd) \cap W^{1,1}_0(\Omega;\Rd)$ with $s \in \set{1,2,3}$ and $p \in [1, \infty]$ we have that
    \begin{align*}
      \norm{ v-\tPihdiv v }_{L^p(T)}  +  h_T \norm{\nabla (v-\tPihdiv v) }_{L^p(T)} &\lesssim h_T^s \norm{\nabla^{s} v}_{L^p(\omega_0(\omega_0(T)))}.
    \end{align*}
  \item  
  \emph{(Continuity)} \label{itm:fortin2-cont} 
  For any $v \in W^{1,p}_0(\Omega)$ with $p \in [1, \infty]$ one has that
  \begin{align*}
    \norm{\nabla \tPihdiv v }_{L^p(T)} &\lesssim \norm{\nabla v}_{L^p(\omega_0(\omega_0(T)))}. 
  \end{align*} 
  \end{enumerate}
    The hidden constants in \ref{itm:fortin2-stab}--\ref{itm:fortin2-cont} depend only on $d$ and the shape regularity of $\tria$.
      If $T \cap \partial \Omega = \emptyset$, then the estimates hold with $\omega_0(\omega_0(T))$ replaced by $\omega_0(\omega_1(T))$. 
\end{proposition}

\begin{proof} 
  To prove the existence of a Fortin operator $\tPihdiv$ we use a modified Scott--Zhang type operator $\overline{\Pi}_1 \colon L^1(\Omega;\setR^d) \to V_h$ and define, with $\Pi_2$ as in \eqref{def:Pi2},
  \begin{align}
    \label{eq:def-Phihdiv2}
    \tPihdiv \coloneqq \overline{\Pi}_1 + \Pi_2(\identity - \overline{\Pi}_1).
  \end{align}
  We employ the locally $L^1$-stable variant of the Scott--Zhang operator as outlined in \cite[p.~491]{SZ.1990}. 
  In this version the averages are taken over $d$-simplices. 

  To avoid difficulties with the approximation properties due to zero traces we start with an $L^1$-stable Scott--Zhang operator $\widetilde{\Pi}_1$ on a regular simplicial partition~$\widetilde{\tria}$ that extends~$\tria$ by one layer of simplices. 
  The domain covered by $\widetilde{\tria}$ is denoted by $\widetilde{\Omega}$. 
  The weighted average integrals associated to each Lagrange node~$\ell$ of~$\mathcal{L}^1_2(\widetilde{\tria})$ are supported on $d$-simplices that contain~$\ell$. 
  For each node $\ell \in \partial \Omega$ on the boundary we choose a $d$-simplex which lies outside of~$\Omega$, i.e., that is contained in~$\widetilde{\tria} \setminus \tria$. 
  Consequently, the operator~$\widetilde{\Pi}_1$ is a projection from~$L^1(\widetilde{\Omega})$ to $\mathcal{L}^1_2(\widetilde{\tria})$ which is locally $L^1$-stable. 
  The specific choice of the $d$-simplices at the boundary ensures that functions $w$ which are zero outside of~$\Omega$ are mapped to discrete functions that are zero outside of~$\Omega$. 

 For every $v \in L^1(\Omega)$ let $\widetilde{v}$ denote the zero extension of~$v$ to $\widetilde{\Omega}$. 
 Note that functions in~$W^{1,1}_0(\Omega)$ extend to functions in~$W^{1,1}_0(\widetilde{\Omega})$.
  We define the linear projection $\overline{\Pi}_1\colon\, L^1(\Omega) \to V_h$ by $\overline{\Pi}_1 v \coloneqq \big(\widetilde{\Pi}_1 (\widetilde{v}\big))|_\Omega$.
  The local $L^1$-stability of~$\widetilde{\Pi}_1$ implies that
  \begin{align}\label{eq:est-localL1}
    \norm{\overline{\Pi}_1 v}_{L^{\infty}(T)} \lesssim \dashint_{\omega_0(T)} \abs{v} \dx \quad \text{ for all } T \in \tria \text{ and } v \in L^1(\Omega).
  \end{align}
  Finally, since $\overline{\Pi}_1 \colon L^1(\Omega;\setR^d) \to V_h$ is a linear projection and by Proposition~\ref{pro:Pi2-div} the operator $\Pi_2 \colon L^1(\Omega;\setR^d) \to V_h$ is linear, we have that the operator $\tPihdiv$  is a linear projection mapping $L^1(\Omega;\setR^d) \to V_h$. 
  Now, the properties~\ref{itm:fortin2-div}--\ref{itm:fortin2-cont} follow as in the proof of Propositions~\ref{pro:fortin-basic} and \ref{pro:fortin-stab-approx} with minor modifications. 
\end{proof}

The corresponding global estimates follow immediately, cf.~Corollary~\ref{cor:fortin-global}. 

\section{Alternative finite element pairs}
\label{sec:alternative}

In this section we discuss variants of the Taylor--Hood element. In Section~\ref{sec:reduced} we present an element with a reduced velocity space using linear functions and tangential edge bubble functions only. In Section~\ref{sec:enriched-taylor-hood} we discuss the opposite situation of the pressure space enriched by piecewise constant functions.

\subsection{Reduced Taylor--Hood element}
\label{sec:reduced}
In this section we introduce an inf-sup stable finite element pair by reducing the velocity pace. 
This is based on the key observation, that the divergence correcting operator in Section~\ref{sec:Pi2} uses only the space~$\edgespace$ spanned by interior tangential edge bubble functions, see~\eqref{eq:edgespace}. 
This allows us to reduce the velocity space choosing
\begin{align*}
  V_h^- \coloneqq \big(\mathcal{L}^1_1(\tria;\mathbb{R}^d) \cap W^{1,1}_0(\Omega)\big) + \edgespace\qquad\text{and}\qquad Q_h \coloneqq \mathcal{L}^1_1(\tria) \cap L^2_0(\Omega).
\end{align*}
Instead of $\Pi_1$ we then use the Scott--Zhang type operator~ $\Pi_1^-\colon W^{1,1}_0(\Omega;\mathbb{R}^d) \to V_h^-$ as described in~\cite{SZ.1990}. 
We  define our Fortin operator $\Pihdiv^-\colon  W^{1,1}_0(\Omega;\mathbb{R}^d) \to V_h^-$ via
\begin{align*}
  \Pihdiv^-\coloneqq \Pi_1^- + \Pi_2(\identity - \Pi_1^-).
\end{align*}
Hence, we may apply the arguments in the proof of Propositions \ref{pro:fortin-basic} and \ref{pro:fortin-stab-approx} to conclude the following result. 
The global estimates follow as before, cf.~Corollary~\ref{cor:fortin-global}.
\begin{proposition}[Fortin operator]
  \label{pro:fortin-stab-approx2}
  We have the following estimates for all $v \in W_0^{1,p}(\Omega;\Rd)$
  with $p \in [1,\infty]$ and all~$T \in \tria$, where the hidden constants depend only on $d$ and the shape regularity of $\tria$. 
  \begin{enumerate}
  \item \label{itm:fortin-basic-div2}
    \emph{(Divergence preserving)}
    For all $v \in W^{1,1}_0(\Omega;\Rd)$ we have that
    \begin{align*}
      \skp{\divergence\,\Pihdiv^- v}{q_h}= \skp{\divergence{v}}{q_h} \qquad \text{ for all }q_h \in Q_h.
    \end{align*}  
  \item \emph{(Local $W^{1,p}$-stability)} One has that \label{itm:fortin-stab2} %
    \begin{align*}
      \lefteqn{\norm{\Pihdiv^- v }_{L^p(T)} + h_T
      \norm{\nabla \Pihdiv^- v }_{L^p(T)}} \qquad
      &
      \\
      &\lesssim \norm{v}_{L^p(\omega_0(\omega_0(T)))}
        + h_T\,\norm{\nabla v}_{L^p(\omega_0(\omega_0(T)))}.
    \end{align*}
  \item \emph{(Approximation)} \label{itm:fortin-approx2} If additionally $v \in W_0^{s,p}(\Omega;\Rd)$ with $s \in \set{1,2}$, we have
    \begin{align*}
      \norm{ v-\Pihdiv^- v }_{L^p(T)}  +  h_T \norm{\nabla (v-\Pihdiv^- v) }_{L^p(T)} &\lesssim h_T^s \norm{\nabla^{s} v}_{L^p(\omega_0(\omega_0(T)))}.
    \end{align*}
  \item  \emph{(Continuity)} \label{itm:fortin-cont2} 
    One has that
    \begin{align*}
      \norm{\nabla \Pihdiv^- v }_{L^p(T)} &\lesssim \norm{\nabla v}_{L^p(\omega_0(\omega_0(T)))}.
    \end{align*}
  \end{enumerate}
  If $T \cap \partial \Omega = \emptyset$, then the estimates hold with $\omega_0(\omega_0(T))$ replaced by $\omega_0(\omega_1(T))$. 
\end{proposition}

In 2D the finite element pair $(V_h^-,Q_h)$ has been used in \cite[p.~542]{MardalSchoeberlWinther2013} in order to construct the Fortin operator for the lowest-order Taylor--Hood element. 
In three and higher dimensions this finite element pair seems to be new. 

The benefit of the reduced element lies in the smaller number of degrees of freedom. 
Naturally, this comes at the cost of a lower approximation rate. 
These two features remind of the MINI element ~\cite[Sec.~8]{BBF.2013}. 
However, the reduced Taylor--Hood element has some advantages compared to the MINI element. 
Since the polynomial degree of the velocity functions  is at most two instead of~$d+1$ for the MINI element, one can use quadrature formulas of lower order. Furthermore, the following consideration shows that the dimensions of the reduced Taylor--Hood space is considerably smaller. 

To showcase the  we compare the degrees of freedom of $(V_h^-,Q_h)$ with the MINI element and the original Taylor--Hood element $(V_h,Q_h)$ in a specific situation. 
More precisely, we consider a standard uniform simplicial partition of the unit cube~$(0,1)^d$ tiled by $N^d$ Kuhn cubes each of which is split into $d!$ Kuhn simplices. 
The partition based on translation of such Kuhn cubes is sometimes referred to as Freudenthal's triangulation, whereas the so-called Whitney--Tucker triangulation arises from translation and reflection of Kuhn cubes, cf.~\cite{WeiFlo2011}.
Both versions lead to the same number of degrees of freedom.
Asymptotically, for large $N$ the triangulations have $N^d$ nodes, $(2^d-1)N^d$ edges and $d!N^d$ $d$-simplices. 
In Table~\ref{fig:compare} we compare the dimensions of the finite element spaces. 
One can see that in 3D the total number of degrees of freedom for the finite elment space pair is reduced from $\approx 22 N^3$ for the MINI element to $\approx 11 N^3$ for the reduced element, which represents a significant reduction.

	\begin{table}[ht!]
  \centering
  \begin{TAB}[3pt]{|c|c|c|c|c|}{|c|ccc|}
    Dimension & $\dim V_h^-$ & $\dim(V_h(\text{MINI}))$ & $\dim V_h$ & $\dim Q_h$
    \\
    2 &  $\approx 5N^2$ & $\approx 6N^2$ & $\approx 8N^2$ & $\approx N^2$
    \\
    3 &  $\approx 10N^3$ & $\approx 21N^3$ & $\approx 24N^3$ & $\approx N^3$
    \\
    d &  $\approx (d+2^d-1)N^d$ & $\approx d(d!+1)N^d$ & $\approx d\,2^d N^d$ & $\approx N^d$
  \end{TAB}
  \caption{Comparison of degrees of freedom}
  \label{fig:compare}
\end{table}

\subsection{\texorpdfstring{The augmented Taylor--Hood and the $P_2$-$P_0$ element}{The augmented Taylor--Hood and the P2-P0 element}}
\label{sec:enriched-taylor-hood}

An extension of the lowest-order Taylor--Hood element is the augmented Taylor--Hood element, for which the discrete spaces are given by
\begin{align*}
  V_h \coloneqq \mathcal{L}^1_2(\tria; \Rd) \cap W^{1,1}_0(\Omega;\Rd)\quad\text{and}\quad
  \Qaug \coloneqq (\mathcal{L}^1_1(\tria) + \mathcal{L}^0_0(\tria)) \cap L^2_0(\Omega).
\end{align*}
A further finite element pair is the $P_2$--$P_0$ element given by the pair $V_h$ and  
\begin{align*}
  Q_h^0 & \coloneqq \mathcal{L}^0_0(\tria) \cap L^2_0(\Omega) \subset \Qaug.
\end{align*}
While both finite element pairs are popular in dimension $d=2$,
computations in $d=3$ show a lack of stability (see for
example~\cite[Sec.~3]{GWW.2014}). 
Since we are not aware of a simple explicit example that underlines this observation we shall present one in this section. 

For our example we use the 3D-octahedron domain. 
We choose the most simple partition with the center as the only interior node. 
More precisely, let~$\Omega \subset \setR^3$ denote the octahedron spanned by the six points $\pm e_m \in \setR^3$, with $e_m$ denoting the $m$-th unit vector.  
Let~$\tria$ denote the partition of~$\Omega$ into eight congruent 3-simplices each of which connects one 2-face of~$\Omega$ with the center~$(0,0,0)$. For simplicity we refer to this setup as \emph{basic partition of the 3D-octahedron}.

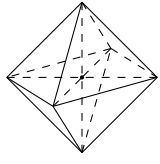
\begin{figure}[ht!]
  \centering
  \begin{tikzpicture}
    x={(0:1cm)},
    y={(45:0.6cm)},
    z={(90:1cm)},
    scale=3.5]
    \coordinate (A) at (0,0,0);
    \draw[dashed] (+1, 0, 0) -- (A);
    \draw[dashed] (-1, 0, 0) -- (A);
    \draw[dashed] ( 0,+1, 0) -- (A);
    \draw[dashed] ( 0,-1, 0) -- (A);
    \draw[dashed] ( 0, 0,+1) -- (A);
    \draw[dashed] ( 0, 0,-1) -- (A);

    \draw (0,0,1) -- (1,0,0);
    \draw (0,0,1) -- (-1,0,0);
    \draw (0,1,0) -- (1,0,0);
    \draw (0,1,0) -- (-1,0,0);
    \draw (0,-1,0) -- (1,0,0);
    \draw (0,-1,0) -- (-1,0,0);
    \draw (0,-1,0) -- (0,0,1);
    \draw (0,1,0) -- (0,0,1);
    \draw[dashed] (0,-1,0) -- (0,0,-1);
    \draw[dashed] (0,1,0) -- (0,0,-1);
    \draw[dashed] (0,0,-1) -- (1,0,0);
    \draw[dashed] (0,0,-1) -- (-1,0,0);
    \fill (0,0,0) circle (.8pt);
  \end{tikzpicture}

  \caption{3D-octahredron~$\Omega$ and its basic partition~$\tria$}
  \label{fig:3Doctahedron}
\end{figure}
The crucial observation is the following.
\begin{proposition}
  \label{pro:P2-P0-counter}
  Let $\tria$ be the basic partition of the 3D-octahedron~$\Omega$. Then the discrete pressure function
  $\bar q_h \in \mathcal{L}^0_0(\tria) \cap L^2_0(\Omega)$ on the
  3D-octrahedron with
  \begin{align*}
    \bar q_h(x) &\coloneqq \sgn(x_1) \sgn(x_2) \sgn(x_3)\qquad\text{for all }x=(x_1,x_2,x_3) \in \Omega
  \end{align*}
  satisfies 
  \begin{align}\label{eq:ZeroFortin}
    \langle \divergence v_h,\bar q_h\rangle = 0 \qquad\text{for all }v_h \in
    \mathcal{L}^1_2(\tria;\setR^3) \cap
    W^{1,1}_0(\Omega;\setR^3).
  \end{align}
\end{proposition}
\begin{proof}
  The space $\mathcal{L}^1_2(\tria;\setR^3) \cap W^{1,1}_0(\Omega;\setR^3)$ is spanned by functions of the form $e_m b$ and $e_m \phi_0$, where $e_m$ is the $m$-th unit vector in $\setR^3$, with $m=1,2,3$, $b$ is an interior edge bubble function and $\phi_0 \in \mathcal{L}^1_1(\tria)$ is the Lagrange basis function associated to the origin.
  In the following we show that the divergence of each of these basis functions is even with respect to at least one of the spatial variables $x_1$, $x_2$, and~$x_3$. 
  Thus, their scalar product with~$\bar q_h$ over~$\Omega$ is zero, which proves~\eqref{eq:ZeroFortin}.

  We start with $\divergence(e_m \phi_0) = e_m \cdot \nabla \phi_0 = -\sgn(x_m)$, which is odd in~$x_m$ but even in the other variables. 
  Now, consider the function $e_m b$, where $b$ is the edge bubble function of the edge $[(0,0,0),(1,0,0)]$, i.e., 
  $b = x_1 (1-x_1-\abs{x_2} -\abs{x_3})\indicator_{\set{x_1 \geq 0}}$. 
  The other cases follow by symmetry. 
  We have that  $\divergence(e_1 \cdot b) = 1-2x_1-\abs{x_2} -\abs{x_3} \indicator_{\set{x_1 \geq 0}}$,
  which is even in~$x_2$ and $x_3$, 
  $\divergence(e_2 \cdot b) = -x_1 \sgn(x_2) \indicator_{\set{x_1 \geq  0}}$, which is even in~$x_3$, and $\divergence(e_3 \cdot b) = -x_1 \sgn(x_3) \indicator_{\set{x_1 \geq 0}}$, which is even in~$x_2$. This proves the claim. 
\end{proof}

Obviously, Proposition~\ref{pro:P2-P0-counter} shows that in 3D the pair of discrete spaces $(V_h, Q_h^0)$
does not satisfy the discrete inf-sup condition, since
\begin{align*}
  0 = \inf_{q_h \in Q_h^0\setminus \lbrace 0 \rbrace} \sup_{v_h \in V_h\setminus \lbrace 0 \rbrace} \frac{\langle \divergence v_h,q_h\rangle}{\lVert \nabla v_h \rVert_{L^2(\Omega)} \lVert q_h \rVert_{L^2(\Omega)}}.
\end{align*}
Hence, there is no linear and bounded operator
$\Pi\colon W^{1,1}_0(\Omega;\mathbb{R}^3) \to V_h$ that preserves the discrete divergence in the sense that
\begin{align*}
  \langle \divergence (v - \Pi v),q_h\rangle = 0\qquad\text{for all }v\in W^{1,1}_0(\Omega;\mathbb{R}^3)\text{ and any }q_h \in Q_h^0.
\end{align*}
The same conclusions certainly also hold for the augmented Taylor--Hood element 
\begin{align*}
   V_h \coloneqq
    \mathcal{L}^1_2(\tria;\setR^3) \cap
  W^{1,1}_0(\Omega;\setR^3)
  \quad\text{and}\quad
  \Qaug \coloneqq
        \big(\mathcal{L}^1_1(\tria) + \mathcal{L}^0_0(\tria)\big) \cap L^2_0(\Omega),
\end{align*}
since only the pressure space is enriched. 
In particular, one cannot extend our Fortin operator for the Taylor--Hood element to the augmented Taylor--Hood element.

\bibliographystyle{plain}
\bibliography{taylor}

\begin{thebibliography}{10}

\bibitem{ABF.1984}
D.~N. Arnold, F.~Brezzi, and M.~Fortin.
\newblock A stable finite element for the {S}tokes equations.
\newblock {\em Calcolo}, 21(4):337--344 (1985), 1984.

\bibitem{BW.2020}
G.~R. Barrenechea and A.~Wachtel.
\newblock The inf-sup stability of the lowest order {T}aylor-{H}ood pair on
  affine anisotropic meshes.
\newblock {\em IMA J. Numer. Anal.}, 40(4):2377--2398, 2020.

\bibitem{Belenki12}
L.~Belenki, L.~C. Berselli, L.~Diening, and M.~R{\r{u}}{\v{z}}i{\v{c}}ka.
\newblock On the finite element approximation of $p$-{S}tokes systems.
\newblock {\em SIAM J. Numer. Anal.}, 50(2):373--397, 2012.

\bibitem{Bo.1997}
D.~Boffi.
\newblock Three-dimensional finite element methods for the {S}tokes problem.
\newblock {\em SIAM J. Numer. Anal.}, 34(2):664--670, 1997.

\bibitem{BBF.2013}
D.~Boffi, F.~Brezzi, and M.~Fortin.
\newblock {\em Mixed Finite Element Methods and Applications}, volume~44 of
  {\em Springer Series in Computational Mathematics}.
\newblock Springer, Heidelberg, 2013.

\bibitem{BCGG.2012}
D.~Boffi, N.~Cavallini, F.~Gardini, and L.~Gastaldi.
\newblock Local mass conservation of {S}tokes finite elements.
\newblock {\em J. Sci. Comput.}, 52(2):383--400, 2012.

\bibitem{BreSco2008}
S.~Brenner and S.~Scott.
\newblock {\em The Mathematical Theory of Finite Element Methods}, volume~15 of
  {\em Texts in Applied Mathematics}.
\newblock Springer, 3rd edition, 2008.

\bibitem{Chen2014}
L.~Chen.
\newblock A simple construction of a {F}ortin operator for the two dimensional
  {T}aylor-{H}ood element.
\newblock {\em Comput. Math. Appl.}, 68(10):1368--1373, 2014.

\bibitem{DieRuzSch10}
L.~Diening, M.~R{\r u}\v{z}i\v{c}ka, and K.~Schumacher.
\newblock A decomposition technique for {J}ohn domains.
\newblock {\em Ann. Acad. Sci. Fenn. Math.}, 35(1):87--114, 2010.

\bibitem{DupSco1980}
T.~Dupont and R.~Scott.
\newblock Polynomial approximation of functions in {S}obolev spaces.
\newblock {\em Math. Comp.}, 34(150):441--463, 1980.

\bibitem{ErnGuermond2004}
A.~Ern and J.-L. Guermond.
\newblock {\em Theory and practice of finite elements}, volume 159 of {\em
  Applied Mathematical Sciences}.
\newblock Springer-Verlag, New York, 2004.

\bibitem{Falk2008}
R.~S. Falk.
\newblock A {F}ortin operator for two-dimensional {T}aylor-{H}ood elements.
\newblock {\em M2AN Math. Model. Numer. Anal.}, 42(3):411--424, 2008.

\bibitem{Feischl19}
M.~Feischl.
\newblock Optimality of a standard adaptive finite element method for the
  {S}tokes problem.
\newblock {\em SIAM J. Numer. Anal.}, 57(3):1124--1157, 2019.

\bibitem{GNS.2005}
V.~Girault, R.~H. Nochetto, and R.~Scott.
\newblock Maximum-norm stability of the finite element {S}tokes projection.
\newblock {\em J. Math. Pures Appl. (9)}, 84(3):279--330, 2005.

\bibitem{GiraultScott2003}
V.~Girault and L.~R. Scott.
\newblock A quasi-local interpolation operator preserving the discrete
  divergence.
\newblock {\em Calcolo}, 40(1):1--19, 2003.

\bibitem{GWW.2014}
B.~Gmeiner, C.~Waluga, and B.~Wohlmuth.
\newblock Local mass-corrections for continuous pressure approximations of
  incompressible flow.
\newblock {\em SIAM J. Numer. Anal.}, 52(6):2931--2956, 2014.

\bibitem{GuzmanSanchez2015}
J.~Guzm\'{a}n and M.~A. S\'{a}nchez.
\newblock Max-norm stability of low order {T}aylor-{H}ood elements in three
  dimensions.
\newblock {\em J. Sci. Comput.}, 65(2):598--621, 2015.

\bibitem{JKN.2018}
V.~John, P.~Knobloch, and J.~Novo.
\newblock Finite elements for scalar convection-dominated equations and
  incompressible flow problems: a never ending story?
\newblock {\em Comput. Vis. Sci.}, 19(5-6):47--63, 2018.

\bibitem{LedererMerdonSchoeberl2019}
P.~L. Lederer, C.~Merdon, and J.~Sch\"{o}berl.
\newblock Refined a posteriori error estimation for classical and
  pressure-robust {S}tokes finite element methods.
\newblock {\em Numer. Math.}, 142(3):713--748, 2019.

\bibitem{MardalSchoeberlWinther2013}
K.-A. Mardal, J.~Sch\"{o}berl, and R.~Winther.
\newblock A uniformly stable {F}ortin operator for the {T}aylor-{H}ood element.
\newblock {\em Numer. Math.}, 123(3):537--551, 2013.

\bibitem{SZ.1990}
L.~R. Scott and S.~Zhang.
\newblock Finite element interpolation of nonsmooth functions satisfying
  boundary conditions.
\newblock {\em Mathematics of Computation}, 54(190):483--493, April 1990.

\bibitem{T.1990}
R.~W. Thatcher.
\newblock Locally mass-conserving {T}aylor-{H}ood elements for two- and
  three-dimensional flow.
\newblock {\em Internat. J. Numer. Methods Fluids}, 11(3):341--353, 1990.

\bibitem{WeiFlo2011}
K.~Weiss and L.~De~Floriani.
\newblock Simplex and diamond hierarchies: Models and applications.
\newblock {\em Computer Graphics Forum}, 30(8):2127--2155, 2011.

\end{thebibliography}

\end{document}